\documentclass{amsproc}

\usepackage{amssymb}

\usepackage{graphicx}

\usepackage[cmtip,all]{xy}

\usepackage[all]{xy}
\usepackage{amsthm}
\usepackage{amscd}
\usepackage{amsmath}
\usepackage{amsfonts}
\usepackage{amssymb}
\usepackage{stmaryrd}
\usepackage{graphicx}%
\usepackage{hyperref}
\hypersetup{
    bookmarks=true,         
    unicode=false,          
    pdftoolbar=true,        
    pdfmenubar=true,        
    pdffitwindow=false,     
    pdfstartview={FitH},    
    pdftitle={My title},    
    pdfauthor={Author},     
    pdfsubject={Subject},   
    pdfcreator={Creator},   
    pdfproducer={Producer}, 
    pdfkeywords={keyword1, key2, key3}, 
    pdfnewwindow=true,      
    colorlinks=false,       
    linkcolor=red,          
    citecolor=green,        
    filecolor=magenta,      
    urlcolor=cyan           
}

\newtheorem{theorem}{Theorem}[section]
\newtheorem{lemma}[theorem]{Lemma}

\theoremstyle{definition}
\newtheorem{definition}[theorem]{Definition}
\newtheorem{example}[theorem]{Example}
\newtheorem{proposition}[theorem]{Proposition}

\theoremstyle{remark}
\newtheorem{remark}[theorem]{Remark}

\numberwithin{equation}{section}

\begin{document}

\title{Almost global homotopy theory}



\author{Zhen Huan}

\address{Zhen Huan, Department of Mathematics,
Sun Yat-sen University, Guangzhou, 510275 China} \curraddr{}
\email{huanzhen84@yahoo.com}
\thanks{The author was partially supported by NSF grant DMS-1406121.}


\subjclass[2010]{Primary 55}

\date{}

\begin{abstract}

In this paper we develop the definition of a global orthogonal spectrum and its unitary version. It relates $G-$equivariant spectra by equivariant weak equivalence in a coherent way. This category of global spectra has a model structure Quillen equivalent to the global model structure on orthogonal spectra. We also show that there is a large family of equivariant cohomology theories, including quasi-elliptic cohomology, that can be globalized in the new context. Starting from one global ring spectrum, we can construct infinitely many distinct global ring spectra. Moreover, in light of the results in this paper, we ask whether we have the conjecture that the globalness of a cohomology theory is completely determined by the formal component of its divisible group and when the $\acute{e}$tale component of it varies the globalness does not change.
\end{abstract}

\maketitle


\section{Introduction}





Equivariant stable homotopy theory is the homotopy theory of topological spaces with group actions. Invented by G. B. Segal in the early 1970s \cite{SegalEquiv}, equivariant
stable homotopy theory was motivated by Atiyah and Segal's work \cite{KC} on equivariant K-theory. The foundation of it was established systematically by tom Dieck, Segal and May afterwards. During the last decades this area has been very active. Equivariant stable homotopy theory has shown great power solving computational and conceptual problems in algebraic and geometric topology. In addition, even early in the history of equivariant stable homotopy theory, people noticed that many equivariant homotopy theories exist not only for one
particular group, but for a family of groups in a uniform way and  as equivariant K-theories, there are relations between  equivariant homotopy theories for different groups, such as the connections embodied in the restriction maps, transfer maps, etc. 
These observations led to the birth of global homotopy theory. Globalness is a measure of the naturalness of a cohomology theory.
Prominent examples of global homotopy theory include global K-theory, global bordism theory and global stable homotopy theory.

Orthogonal spectrum is a good model to interpret global spectrum, which reflects most symmetry. The idea of global orthogonal spectra was first inspired in
Greenlees and May \cite{GP}.
Several models of
global homotopy theory have been established, which are equivalent to each other. Schwede develops a modern approach to global homotopy theory in \cite{SS}. In \cite{BG}
Bohmann uses the framework of enriched indexed categories to develop the definition of a global orthogonal spectrum. She shows that the Atiyah-Bott-Shapiro orientation can extend
to this global context. Gepner and Henriques establish the basic theory of an unstable global homotopy theory in \cite{Gepnerorbi} via infinity categories. The underlying infinity category is equivalent to that of Schwede's orthogonal spaces. This theory is much easier to work with for elliptic cohomology. 

The motivating example for almost global homotopy theory is quasi-elliptic cohomology \cite{Rez11}\cite{Huanthesis}\cite{Huansurvey}. This theory is a variant of
elliptic cohomology theories. It is the orbifold K-theory of a space of constant
loops. For global quotient orbifolds, it can be expressed in terms
of equivariant K-theories and has the change-of-group isomorphisms. In a conversation, Ganter indicated
that it has better chances than Grojnowski equivariant elliptic
cohomology theory to be put together naturally in a uniform way
and made into an ultra-commutative global cohomology theory in the
sense of Schwede \cite{SS}. However, though many constructions on quasi-elliptic cohomology can be made in a neat way,
there is not an effective way to prove or disprove this statement.
We construct in \cite{HuanSpec} an orthogonal $G-$spectrum of quasi-elliptic cohomology for each compact Lie group $G$. As indicated in
Remark 6.13 \cite{HuanSpec}, these equivariant orthogonal spectra cannot arise from an orthogonal spectrum.
Then it is even more difficult to see whether each elliptic cohomology theory,
whose form is more intricate and mysterious than quasi-elliptic cohomology, can be globalized in the current setting.

In this paper we construct a new theory, almost global homotopy theory. We observe that if the equivariant cohomology theories $\{E_G^*(-)\}_G$ have the change-of-group isomorphisms and there exists
an orthogonal $G-$spectrum  $\{X_G(-), \sigma_{(-,-)}\}$ representing $E_G^*(-)$, then for any closed subgroup $H$ of $G$ and any $G-$representation $V$,
we have a weak $H-$equivalence $X_G(V)\simeq_H X_H(V)$. In light of this observation, we combine the equivariant orthogonal spectra of $\{E_G^*(-)\}_G$ by equivariant weak equivalences to give the definition of a global spectrum. 
The category $Sp_W^{D_0}$ [Definition \ref{spd0w}] of the resulting global spectra has a model structure on it Quillen equivalent to the global model structure on orthogonal spectra constructed in
Theorem 4.3.17 \cite{SS}.

Other than quasi-elliptic cohomology, there is a large family of equivariant cohomology theories that can extend to the new context of global spectra.
In \cite{Huanquasi}
we define quasi-theories. 
They are motivated by quasi-elliptic cohomology whose divisible group is given by the Tate curve $Tate(q)$ [Section 2.6,
\cite{AHS}]. Starting from a cohomology theory $E$, we can construct a family of theories $QE_{n, G}^*(-)$ for any positive integer $n$ and any compact Lie group $G$.
Quasi-elliptic cohomology, the $n-$th generalized quasi-elliptic cohomology and the theories $QE^*(-)$ in \cite{HuanSpec} are all quasi-theories.

Quasi-theories $\{QE_{n, G}^*(-)\}_G$ have the change-of-group isomorphism if the initial theory $\{E^*_G(-)\}_G$ does, as shown in \cite{Huanquasi}. In \cite{Huanquasisp} we construct an orthogonal $G-$spectrum for each quasi-theory. The construction generalizes that for quasi-elliptic cohomology in \cite{HuanSpec}. The equivariant orthogonal spectra for each quasi-theory cannot arise from an orthogonal spectrum, but do form a global spectra in the almost global homotopy theory.

In this paper we construct almost global homotopy theory and several model structures on the underlying category. In addition, we construct global spectra for quasi-theories
in the new setting. 
In Section \ref{newdiagram} we construct the new diagram category $D_0$ by adding restriction maps between group representations to the category $\mathbb{L}$ of
inner product vector spaces. 
In Section \ref{d0tw} we introduce the category $Sp_W^{D_0}$ of global spectra, which is a full subcategory of the category of $D_0-$spectra.
Moreover, we introduce the unitary version of it.
In Section \ref{modelgl} we describe the global model structure on $Sp^{D_0}_W$ and show it is Quillen equivalent to
the global model structure on orthogonal spectra
in Theorem 4.3.18 \cite{SS}.
In Section \ref{dfspmodel}, for any global family $\mathcal{F}$, we introduce a category $Sp_W^{D_0^{\mathcal{F}}}$ of global spectra and show that if $\mathcal{F}$
is the global family of finite groups there is a model structure on $Sp_W^{D_0^{\mathcal{F}}}$ Quillen equivalent to the $\mathcal{F}-$global model structure on orthogonal spectra
in Theorem 4.3.17 \cite{SS}. In addition, we describe roughly the Reedy model structure on $Sp_W^{D_0^{\mathcal{F}}}$ constructed in Section 6.4 \cite{Huanthesis}.
In Section \ref{examples} we show quasi-theories have global version in this new context.
At the end of the paper, in Remark \ref{etaless}, we raise the question whether the $\acute{e}$tale component in the divisible group associated to a cohomology theory plays
no role in the globalization of the theory.

\subsection{Acknowledgement}
First I would like to thank my PhD advisor Charles Rezk. The main idea of this paper is motivated by him.
I would also like to thank Matthew Ando for his support and encouragement  all the time. I would like to thank Stefan Schwede for helpful conversation at Homotopy Theory Summer Berlin
in June 2018 and a few email discussions. I would like to thank David Gepner and Nathaniel Stapleton for helpful conversations.

\section{The categories $D_0$ and $D_0^{\mathbb{C}}$}\label{newdiagram}

In this section we introduce the new diagrams $D_0$ and $D_0^{\mathbb{C}}$.
\subsection{The category $D_0$}

\begin{definition}
Let $D$ be a category with objects $(G, V,\rho)$ where $V$ is an
inner product vector space, $G$ a compact group $G$ and $\rho$ a
faithful group representations
$$\rho: G\longrightarrow O(V).$$ A morphism in $D$ $$\phi=(\phi_1, \phi_2): (G, V,
\rho)\longrightarrow (H, W, \tau)$$ consists of a linear isometric
embedding $\phi_2: V\longrightarrow W$ and a group homomorphism
$\phi_1: \tau^{-1}(O(\phi_2(V)))\longrightarrow G$, which
make the diagram commute. \begin{equation}\xymatrix{G\ar[r]^{\rho}
&O(V)\ar[d]^{\phi_{2*}}\\
\tau^{-1}(O(\phi_2(V)))\ar[u]^{\phi_1}\ar[r]^>>>>>{\tau}
&O(W)}\label{Dmor}\end{equation} In other words, the group action
of $H$ on $\phi_2(V)$ is induced from that of $G$.

The composition of two morphisms
$$(G, V, \rho)\buildrel{(\phi_1, \phi_2)}\over\longrightarrow (H, W, \tau)\buildrel{(\psi_1, \psi_2)}\over\longrightarrow (K, U,
\beta)$$ is defined to be
$$(\phi_1\circ\psi_1|_{\beta^{-1}(O(\psi_2\circ\phi_2(V)))},
\psi_2\circ\phi_2).$$
The composition is associative.
The identity morphism in $D((G, V, \rho),
(G, V, \rho))$ is $(Id, Id)$.

\end{definition}
All the maps in (\ref{Dmor}) are injective. Given a linear
isometric embedding $\phi_2: V\longrightarrow W$,
$\tau^{-1}(\phi_{2*}(\rho(G)))$ is always nonempty since the
identity element is in it.  If $\phi_1$ exists, it is unique.

\begin{lemma}Two objects $(G, V, \rho)$ and $(H, W, \tau)$ in $D$ are
isomorphic if and only if  there is an isomorphism
$G\longrightarrow H$ which makes $V$ and $W$ isomorphic as
representations. \label{isoobd}
\end{lemma}
The proof is left to the readers.

\begin{remark}This category $D$ has the same objects as the category $\mathcal{I}_{\mathcal{G}}$ defined in Example 2.3, \cite{BG}, from which Bohmann constructed
a version of global spectra. However, the sets of morphisms in the two categories are very different. Roughly speaking, $D$ has the same restriction maps as $\mathcal{I}_{\mathcal{G}}$
and less linear isometric embeddings than  $\mathcal{I}_{\mathcal{G}}$. In addition, there is no obvious structure of enriched indexed category on $D$.
\end{remark}

\begin{remark}
Since the representation $\rho$ in an object $(G, V, \rho)$ of $D$
is faithful, we may consider the group $G$ as a closed subgroup of
$O(V)$. Then the diagram (\ref{Dmor}) is equivalent to

\begin{equation}\xymatrix{G\ar[r] &O(V)\ar[d] \\ H\cap O(V)\ar[u] \ar[r] &O(W)}\label{simplifyD}\end{equation}
where we consider $O(V)$ as a closed subgroup of $O(W)$ as well.
All the maps in the diagram are inclusions.

\end{remark}

In Section 6.1 \cite{Huanthesis}, we discuss the topology on the space $D((G, V, \rho), (H, W, \tau))$ of morphisms.
In addition, the space $D((G, V, \rho), (H, W, \tau))$ of morphisms inherits an
$H-$action and $G-$action on it. For
$\phi=(\phi_1,\phi_2): (G, V, \rho)\longrightarrow (H, W, \tau)$
and $(g, h)\in G\times H$,

\begin{equation}(g, h)\cdot (\phi_1,\phi_2) = (g\phi_1(h^{-1}-h)g^{-1},
h\phi_2(g^{-1}\cdot -))\label{baction}\end{equation} The
$H-$action on $\phi_2$ is left and the $G-$action on it is right,
whereas the $H-$action on $\phi_1$ is right and the $G-$action on
$\phi_1$ is left.

\begin{proposition}
The category $D$ is a symmetric monoidal category.
\label{d0symmetric}
\end{proposition}

\begin{proof}
The tensor product $+: D\times D\longrightarrow D$ is defined by
\begin{equation}((G,V,\rho), (H,W,\tau))\mapsto (G\times H,
V\oplus W, \rho\oplus\tau).\label{tensorp}\end{equation} The unit
object is $u=(e, 0, \ast)$ where $e$ is the trivial group and  $\ast$ is the 
unique map from $e$. From the property of product of
representations, the tensor product is associative. And we have
the isomorphism $(G\times H, V\oplus W,
\rho\oplus\tau)\longrightarrow (H\times G, W\oplus V,
\tau\oplus\rho)$.

It's straightforward to check it satisfies all the required
diagrams.

\end{proof}

Note that each object $(G, V, \rho)$ is isomorphic to $(\rho(G),
V, i)$ in $D$ with $i: \rho(G)\hookrightarrow O(V)$ the inclusion of
closed subgroup.  Let  $D_0$ denote the full subcategory of $D$ consisting of objects $(G, V, i)$ with $G$ a
closed subgroup of $O(V)$ and $i$ the inclusion. $D_0$ is equivalent to $D$. We will not lose any information if we study $D_0$ instead of $D$. We use $(G, V)$ to denote
$(G, V, i)$ in the rest of the paper.

Next we introduce two types of morphisms in $D_0$, the restriction maps and the linear isometric embeddings.
\begin{definition}A morphism $\phi=(\phi_1, \phi_2): (G, V)\longrightarrow (H, W)$ in $D_0$ is called a restriction map if $\phi_2$ is an isomorphism. \end{definition}

\begin{definition} A morphism $\phi=(\phi_1, \phi_2): (G, V)\longrightarrow (H, W)$ in $D_0$ is called a linear isometric embedding if $\phi_1$ is an isomorphism.\end{definition}

\begin{proposition} Each morphism $\phi=(\phi_1, \phi_2): (G, V)\longrightarrow (H, W)$ in $D_0$
is the composition of a restriction map and a linear isometric embedding up to isomorphisms. \end{proposition}
\begin{proof}Any morphism $(\phi_1, \phi_2): (G, V)\longrightarrow (H, W)$ has
the decomposition \begin{equation}(G, V)\buildrel{(i,
Id)}\over\longrightarrow (\phi_1(H\cap O(\phi_{2}(V))),
V)\buildrel{(\phi_1, \phi_2)}\over\longrightarrow (H,
W)\label{dec1o}\end{equation} where $i:
\phi_1(H\cap O(\phi_{2}(V)))\longrightarrow G$ is the inclusion. Note
that in the second morphism $(\phi_1, \phi_2)$, $$\phi_1:
H\cap O(\phi_{2}(V))\longrightarrow \phi_1(H\cap O(\phi_{2}(V)))$$ is a
group isomorphism.

If $(\phi_1, \phi_2)=(f_1, f_2)\circ (\alpha_1, \alpha_2)$ with
$(\alpha_1, \alpha_2): (G, V)\longrightarrow (G', V')$ a restriction map
and $(f_1, f_2): (G', V')\longrightarrow (H, W)$ a linear isometric embedding,
$\alpha_2: V\longrightarrow V'$ is an isometric isomorphism and
$f_1: H\cap O(f_{2}(V'))\longrightarrow G'$ is a group
isomorphism. The group homomorphisms $\alpha_1$ (resp. $f_1$) is
uniquely determined by $\alpha_2$ (resp. $f_2$). Note that
$O(f_{2}(V'))\cap H=O(\phi_2(V))\cap H.$ We have the commutative
diagram below.
$$\xymatrix{\phi_1(H\cap O(\phi_2(V)))
\ar[r]&O(V)\ar[d]^{\alpha_2}\\
G'\ar[u]^{\phi_1\circ f_1^{-1}}\ar[r] &O(V')}$$ So we have the
isomorphism $$(\phi_1\circ f_1^{-1}, \alpha_2):
(\phi_1(H\cap O(\phi_2(V))), V)\longrightarrow (G', V').$$  The
morphism $(i, Id): (G, V)\longrightarrow (\phi_1(H\cap
O(\phi_2(V))), V)$ equals to the composition $(\phi_1\circ
f_1^{-1}, \alpha_2)^{-1}\circ (\alpha_1, \alpha_2)$ and $(\phi_1,
\phi_2):(\phi_1(H\cap O(\phi_2(V))), V)\longrightarrow (H, W,
\tau)$ is the composition $(f_1, f_2)\circ (\phi_1\circ f_1^{-1},
\alpha_2)$. So the decomposition (\ref{dec1o}) is unique up to
isomorphism.

\end{proof}

\subsection{The category $D^{\mathbb{C}}$}

Moreover, we can define the unitary version of $D$ and $D_0$.
Let $W$ be a complex inner product space, i.e. a finite
dimensional $\mathbb{C}-$vector space equipped with a hermitian
inner product $(-, -)$. Let $rW$ denote the underlying real inner
product spaces of $W$, i.e. the underlying finite dimensional
$\mathbb{R}-$vector space equipped with the euclidean inner
product $$\langle v, w\rangle = \mbox{Re}(v, w).$$ Let
$$C=\mbox{Gal}(\mathbb{C}/\mathbb{R})=\{\mbox{Id}_{\mathbb{C}},
\tau\}$$ denote the Galois groups of $\mathbb{C}$ over
$\mathbb{R}$ where $\tau: \mathbb{C}\longrightarrow \mathbb{C}$ is
the complex conjugation $\tau(\lambda)=\overline{\lambda}.$

First we recall some constructions from \cite{SS}.
\begin{definition}$L^{\mathbb{C}}$ is the complex isometries category
whose objects are finite dimensional complex inner product spaces
and morphism space between two objects $V$ and $W$ is the space of
the pairs $$(\psi, \sigma)\in L(rV, rW)\times C$$ that satisfy
$\psi(\lambda\cdot v)=\sigma(\lambda)\cdot\psi(v)$ and $(\psi(v),
\psi(v'))=\sigma((v,v')).$ The composition in $L^{\mathbb{C}}$ is
defined by $$(\psi, \sigma)\circ (\psi', \sigma')=(\psi\psi',
\sigma\sigma')$$ and the identity morphism  of $V$ is
$(\mbox{Id}_V,
\mbox{Id}_{\mathbb{C}})$.\label{lcunitarydef}\end{definition}
The extra piece of the unitary structure gives a richer global
homotopy theory that is indexed not only on compact Lie groups,
but on the larger class of augmented Lie groups.

\begin{definition}An augmented Lie group is a compact Lie group $G$ equipped with a continuous homomorphism $\epsilon: G\longrightarrow C$, called the augmentation,
to the Galois group of $\mathbb{C}$ over
$\mathbb{R}$.\end{definition}

Let $G_{ev}=\epsilon^{-1}(Id_{\mathbb{C}})$ denote the even part
of $G$ and $G_{odd}=\epsilon^{-1}(\tau)$ the odd part of $G$. The
product in the category of augmented Lie groups is defined to be
the fiber product  over $C$. Explicitly, the product of two
augmented Lie groups  $G$ and $K$ is the augmented Lie group
$G\times_C K$ with $(G\times_C K)_{ev}=G_{ev}\times K_{ev}$ and
$(G\times_C K)_{odd}=G_{odd}\times K_{odd}$.

\begin{example}[extended unitary group]The endomorphism group of a complex
inner product space $W$ in the category $L^{\mathbb{C}}$ is
defined to be $$\widetilde{U}(W):= L^{\mathbb{C}}(W, W).$$ The
augmentation $\epsilon_W: \widetilde{U}(W)\longrightarrow C$ is
defined to be $\epsilon_W(\varphi, c)=c$.
\end{example}

The extended unitary group is a closed subgroup of $O(rW)$.

The augmented Lie group contain compact Lie groups as the ones
with trivial augmentation. It also contain the Real Lie groups of
Atiyah and Segal in  \cite{KC}, which are defined as compact Lie
groups equipped with an involution.

\begin{example}[Split augmented Lie groups] Let $G$ be a compact 
Lie group equipped with an involution $\tau: G\longrightarrow G$
on it, namely a Real Lie group in the sense of \cite{KR}. We can
construct an augmented Lie group from it. The semi-direct product
$G\rtimes_{\tau} C$ is an augmented Lie group with the
augmentation
$$G\rtimes_{\tau} C\longrightarrow C, \mbox{   }  (g,
\sigma)\mapsto\sigma.$$

A real representation of $G\rtimes_{\tau} C$ amounts to a unitary
representation $V$ of $G$ with a real structure $\tau:
V\longrightarrow V$ such that $$\tau(g\cdot v)=\tau(g)\cdot
\tau(v), \mbox{    }\forall g\in G, v\in V.$$

For the opposite direction, given an augmented Lie group,  it is
isomorphic to a $G\rtimes_{\tau} C$ for some Real Lie group $G$ if
and only if its augmentation has a multiplicative section, i.e.,
it has an odd element of order 2.

For any complex inner product space $W$, there is a canonical
involution $\tau$ on $U(W)$ by complex conjugation.  We have the
isomorphism of augmented Lie groups
\begin{equation}U(W)\rtimes_{\tau}
C\longrightarrow \widetilde{U}(W),\mbox{   } (\psi, \tau)\mapsto
\psi\circ\tau_W.\end{equation}

\label{split}
\end{example}

\begin{definition}A real representation of an augmented Lie group $G$ is a finite-dimensional complex inner product space $V$ and a continuous homomorphism
$\rho: G\longrightarrow\widetilde{U}(V)$, i.e., such that
$\epsilon_V\circ\rho=\epsilon$.\end{definition}

\begin{definition}Let $G$ be an augmented Lie group. An augmented right $G-$space is a right $G-$space $A$ equipped with a continuous map $\epsilon: A\longrightarrow C$
such that $$\epsilon(a\cdot g)=\epsilon(a)\cdot\epsilon(g)$$ for
all $a\in A$ and all $g\in G$.\end{definition}

\begin{definition} A unitary space is a continuous functor from the complex isometries category $L^{\mathbb{C}}$ to the category of spaces, where $L^{\mathbb{C}}$ is the category
in Definition \ref{lcunitarydef}. A morphism of unitary spaces is
a natural transformation of functors. We denote the category of
unitary spaces by $spc^U$.
\end{definition}

\begin{definition}$D^{\mathbb{C}}$ is the category whose objects are  $(G, V, \rho)$ with $G$ an augmented Lie group and
$(V, \rho)$ a faithful real representation of $G$, and the
morphism space $D^{\mathbb{C}}((G, V, \rho), (H, W, \tau))$ is the space of the
pairs $(\phi_1, \phi_2)$ with $\phi_2\in L^{\mathbb{C}}(V, W)$ and
$\phi_1: \tau^{-1}(\widetilde{U}(\phi_2(V)))\rho(G)\longrightarrow
G$ a group homomorphism, which make the diagram commute.
\begin{equation}\xymatrix{G\ar[r]^{\rho}
&\widetilde{U}(V)\ar[d]^{\phi_{2*}}\\
\tau^{-1}(\widetilde{U}(\phi_2(V)))\ar[u]^{\phi_1}\ar[r]^>>>>>{\tau}
&\widetilde{U}(W)}\label{DmorU}\end{equation} In other words, the
action of the augmented Lie group $H$ on $\phi_2(V)$ is induced from that of $G$.
\end{definition}

\begin{proposition}The category $D^{\mathbb{C}}$ is a symmetric monoidal category.\end{proposition}
\begin{proof}
Let $(G, V, \rho)$ and $(H, W, \tau)$ be two objects in
$D^{\mathbb{C}}$. The tensor product of $(G, V, \rho)$ and $(H, W,
\tau)$ is defined to be $(G\times_C H, V\oplus W,
\rho\times_{C}\tau)$. The product is obviously associative and
commutative. The unit is $(1, \empty, 1)$ where $1$ is the trivial
group equipped with the trivial augmentation.

\end{proof}

\begin{definition}The category $D^{\mathbb{C}}_0$ is defined to be the full subcategory of $D^{\mathbb{C}}$ consisting of objects $(G, V)$ with $G$ a closed subgroup of
$\widetilde{U}(V)$.  It is equivalent to $D^{\mathbb{C}}$ as categories. \end{definition}

\section{The Category $Sp_W^{D_0}$}\label{d0tw}

In this section we define the category $Sp_W^{D_0}$ of global spectra that we will work in. At the end of this section, we
introduce the unitary version of it.

 We implicitly add disjoint basepoints to the morphism spaces in $D_0$.

\begin{definition}[The category $D_0T$]A $D_0-$space is
a continuous functor $X: D_0\longrightarrow T$ to the category of based
compactly generated weak Hausdorff spaces. A morphism of
$D_0-$spaces is a natural transformation. We use $D_0T$ to
denote the category of $D_0-$spaces.\end{definition}

\begin{example}

We can define the $D_0-$sphere. For each object $(G, V)$ in $D$,
define \begin{equation}S^{(G, V)}:=S^V.   \end{equation} $S^V$
inherits a $G-$action from that on $V$.

For any morphism $\phi=(\phi_1, \phi_2): (G, V)\longrightarrow (H, W)$ in $D_0$. Define $S(\phi)$ to be
\begin{equation}S^{\phi_2}:
S^V\longrightarrow S^W.
\end{equation}
In particular,   $S(Id)$ is the identity map.\label{d0sphere}\end{example}


\begin{definition}[The category $Sp^{D_0}$] 
A $D_0-$spectrum $X$ consists of
\begin{itemize} \item a based $G-$space $X(G, V)$;
\item a based structure map $\sigma_{(G, V), (H, W)}: S^W\wedge X(G, V)\longrightarrow X(G\times H, V\oplus W)$ which is $G\times H-$equivariant\end{itemize}
for any objects $(G, V)$ and $(H, W)$ of $D_0$.

In addition, $\sigma_{(G, V), (1, 0)}$ is the identity map and $$\sigma_{(G, V), (H_1\times H_2, W_1\oplus W_2)}=\sigma_{(G\times H_1, V\oplus W_1), (H_2, W_2)}\circ (S^{W_2}\wedge \sigma_{(G, V), (H_1, W_1)}).$$

A morphism of $D_0-$spectra  $f: X\longrightarrow Y$ is a functor in $D_0T$ compatible with the structure maps, i.e. for any objects $(G, V)$ and $(H, W)$ in $D_0$, $$f(G\times H, V\oplus W)\circ \sigma^{(X)}_{(G, V), (H, W)}=\sigma^{(Y)}_{(G, V), (H, W)}\circ (f(G, V)\wedge S^W).$$

We use $Sp^{D_0}$ to denote the category of $D_0-$spectra. \label{spd0}
\end{definition}

\begin{remark}A $D_0-$spectrum can be interpreted as a diagram spectrum. First we give the structure maps a topology. 
Let $(G, V)$ and $(H, W)$ be two objects in $D_0$.
Over the space $D_0((G, V), (H, W))$ of morphisms we define a vector bundle with total space
$$\xi((G, V), (H, W))=\{(w, (\phi_1, \phi_2))\in W\times D_0((G, V), (H, W)) | w \perp \phi_2(V)\}$$
and with the structure map $\xi((G, V), (H, W))\longrightarrow D_0((G, V), (H, W))$ the projection to the second factor.
Let $\textbf{O}((G, V), (H, W))$ denote the Thom space of this bundle. 
We have \begin{equation}\textbf{O}((G, V), (H, W))\wedge V\cong W\wedge D_0((G, V), (H, W)). \end{equation}
A $D_0-$spectrum is a continuous based functor from $\textbf{O}$ to the category \textbf{T} of based compactly generated weak Hausdorff spaces.

\end{remark}

To study equivariant homotopy theory, a more interesting and reasonable category is 
defined below.
\begin{definition}[The category $D_0T^W$]Define $D_0T^W$ to be the full category of $D_0T$ consisting of those objects $X: D_0\longrightarrow T$ that maps each restriction map $(G, V)\longrightarrow (H, V)$ to an $H-$weak equivalence. We call the objects $D_0^W-$spaces. \label{defd0wt} \end{definition}

\begin{definition}
[The category $Sp_W^{D_0}$] A $D_0^W-$spectrum $X$ is both a $D_0-$spectrum and a $D_0-$space in $D_0T^W$. The category $Sp_W^{D_0}$ is the full
subcategory of $Sp^{D_0}$ consisting of $D_0^W-$spectra.
\label{spd0w}\end{definition}

We recall the definition of orthogonal spectra.
\begin{definition}[The category $Sp^O$ of orthogonal spectra]
An orthogonal spectrum consists of the following data: \begin{itemize}
\item a sequence of pointed spaces $X_n$ for $n\geq 0$
\item a base-point preserving continuous left action of the orthogonal group $O(n)$ on $X_n$ for each $n\geq 0$
\item based maps $\sigma_n: X_n\wedge S^1\longrightarrow X_{n+1}$ for $n\geq 0$.\end{itemize}
This data is subject to the following condition: for all $n, m\geq 0$, the iterated structure map $$\sigma^m: X_n\wedge S^m\longrightarrow X_{n+m}$$
is $O(n)\times O(m)-$equivariant.

We use $Sp^O$ to denote the category of orthogonal spectra.

\label{orthospec}
\end{definition}


Since $D_0$ is a topological category, $Sp^{D_0}$ is complete and cocomplete.
\begin{lemma}$Sp^{D_0}_W$ is both complete and cocomplete. \label{cococo}\end{lemma}
\begin{proof}

We first show that $Sp_W^{D_0}$ is complete.

Let $J$ be a small category and $F: J\longrightarrow Sp_W^{D_0}$ be a diagram. Since $Sp^{D_0}$ is complete, the diagram $F: J\longrightarrow Sp_W^{D_0}$ has a limit $X$
in $Sp^{D_0}$. Let $r: (G, V)\longrightarrow (H, V)$ be a restriction map. Let $ev_{(G, V)}: Sp_W^{D_0}\longrightarrow T$ be the evaluation map sending $X$ to $X(G, V)$.
Then we have diagrams $$ev_{(G, V)}\circ F: J\longrightarrow T\mbox{  and  }ev_{(H, V)}\circ F: J\longrightarrow T.$$ The limits of them are $X(G, V)$ and $X(H, V)$ respectively.
For each integer $n$, we have the equivariant
homotopy groups $\pi_n^H(-)$. The diagrams $$\pi_n^H\circ ev_{(G, V)}\circ F\mbox{  and  }\pi_n^H\circ ev_{(H, V)}\circ F$$ are isomorphic. So we have isomorphic limits
$$\pi_n^H(X(G, V))\cong \pi_n^H(X(H, V))\mbox{   for each }n.$$  Thus, $X(G, V)$ and $X(H, V)$ are $H-$weak equivalent. So $X$ is an object in $Sp_W^{D_0}$.

\bigskip
We can prove that $Sp_W^{D_0}$ is cocomplete in a similar way.

\end{proof}

We have a pair of adjoint functors between $Sp_W^{D_0}$ and the category $Sp^O$ of orthogonal spaces. In Section 6.1 \cite{Huanthesis} we discuss a functor $$l: L\longrightarrow
D$$ sending $\mathbb{R}^n$ to $(O(n), \mathbb{R}^n)$. It gives a forgetful functor $U: Sp^{D_0}\longrightarrow Sp^{O}$.
Let $$Q: Sp_W^{D_0}\longrightarrow Sp^{D_0}\longrightarrow Sp^{O}$$ denote the composition of the inclusion and the forgetful functor $U$.

In addition, we can define a functor $P: Sp^{O}\longrightarrow Sp_W^{D_0}$ by sending an orthogonal spectrum $Y$ to the $D_0^W-$spectrum \begin{equation}
P(Y)(G, V):= Y(V); P(Y)(\phi_1, \phi_2):=Y(\phi_2) \label{leftP}\end{equation} for any object $(G, V)$  and any morphism $(\phi_1, \phi_2)$ in $D_0$.

\begin{proposition}

\begin{displaymath}
(P \dashv Q) : Sp^{O} \stackrel{\overset{Q}{\leftarrow}}{\underset{P}{\to}}  Sp_W^{D_0}
\end{displaymath}
is a pair of adjoint functors. In addition, the unit $\eta: Id_X\longrightarrow Q\circ P$ is an isomorphism.
\label{qrtaj}\end{proposition}

\begin{proof}
Let $X$ be an object in $Sp^O$ and $Y$ an object in $Sp_W^{D_0}$. We can define a natural isomorphism $F: Hom_{Sp^O}(-, Q(-))\longrightarrow Hom_{Sp_W^{D_0}}(P(-), -)$.

For any functor $f: X\longrightarrow Q(Y)$, let $F(f)(G, V): P(X)(G, V)\longrightarrow Y(G,V)$ be defined by the composition
\begin{equation}\xymatrix{P(X)(O(V), V) \ar[d]_{=} \ar[r]^{f(V)} &Y(O(V), V) \ar[d]_{Y(i, Id)} \\
P(X)(G, V)\ar[r] &Y(G, V)}\end{equation}for any object $(G, V)$ of $D_0$.

In addition, we can define $\alpha: Hom_{Sp_W^{D_0}}(P(-), -)\longrightarrow Hom_{Sp^O}(-, Q(-))$. For any functor $\beta: P(X)\longrightarrow Y$, define
$\alpha(\beta)(V):=\beta(O(V), V)$. 
$\alpha$ is the inverse of $F$.

The composition $Q\circ P$ is the identity. So $\eta$ is an isomorphism.
\end{proof}

\begin{remark}
The adjoint functors $P \dashv Q$ are not equivalences. $P$ is full and faithful but it is not essentially surjective.
 \end{remark}

\begin{example}
Each orthogonal spectrum $Y$ gives a $D_0^W-$spectrum $QY$.

\end{example}

We can also define the equivariant homotopy groups of $D_0^W-$spectrum.

\begin{definition}For each $D_0^W-$spectrum $Y$, \begin{equation}
\pi^G_k(Y):=\pi^G_k(Q(Y)), \label{htpygrpd0}
\end{equation}for any compact Lie group $G$ and any positive integer $k$. \label{globalhtpygrp}\end{definition}

\begin{remark}Note that for any $D_0^W-$spectrum, each restriction map in $D_0$ is mapped to an equivariant weak equivalence. Thus,
the definition in (\ref{htpygrpd0}) contains all the necessary information.

In addition, we can define the equivariant homotopy groups of $D_0^W-$spectrum in the canonical way as those of orthogonal spectra. We describe the construction
explicitly. Let $s(\mathcal{U}_G)$ denote the poset, under inclusion of finite-dimensional $G-$subrepresentations of the chosen complete $G-$universe $\mathcal{U}_G$.
Let $X$ denote a $D_0^W-$spectrum $Y$. If $\psi=(Id, \psi_2): (G, V)\longrightarrow  (G, W)$ is a morphism in $D_0$ and
$f: S^V\longrightarrow X(G, V)$ a continuous based map, we define
$\psi_*f: S^W\longrightarrow X(G, W)$ as the composite
$$\begin{CD}\psi_*f: S^W\cong S^V\wedge
S^{W-\psi_2(V)}  @>{f\wedge S^{W-\psi_2(V)}}>>   X(G, V)\wedge
S^{W-\psi_2(V)}  @>{\psi_{2*}}>>  X(G, W).\end{CD}$$

We obtain a functor from the poset
$s(\mathcal{U}_G)$ to sets by sending $V\in s(\mathcal{U}_G)$ to
$$[S^V, X(G, V)]^G,$$ the set of $G-$equivariant homotopy classes of
based $G-$maps from $S^V$ to $X(G, V)$.  For $V\subseteq W$ in
$s(\mathcal{U}_G)$, the inclusion $i: V\longrightarrow W$ is sent
to the map $$[S^V, X(G, V)]^G\longrightarrow [S^W, X(G, W)]^G, \mbox{
}[f]\mapsto [i_*f].$$

The 0-th equivariant homotopy group $\pi^G_0(X)$ is then defined
as $$\pi^G_0(X)=\mbox{colim}_{V\in s(\mathcal{U}_G)}[S^V,
X(G, V)]^G.$$

Let $k$ be a positive integer. Define
\begin{align}\pi^G_k(X)&:=colim_{V\in s(\mathcal{U}_G)}[S^{V\oplus \mathbb{R}^k}, X(G, V)]^G \\ \pi^G_{-k}(X)&:=
colim_{V\in s(\mathcal{U}_G)}[S^{V},  X(G, V\oplus \mathbb{R}^k)]^G.\end{align}
\end{remark}

Moreover, we define and study the unitary version of $Sp_W^{D_0}$.

\begin{definition} A $D_0^{\mathbb{C}}-$space is a continuous functor from the category $D_0^{\mathbb{C}}$ to the category of spaces.
A morphism of $D_0^{\mathbb{C}}-$spaces is a natural
transformation of functors. We denote the category of
$D_0^{\mathbb{C}}-$spaces by $D_0^{\mathbb{C}}T$.
\end{definition}

$D^{\mathbb{C}}_0-$spectra are the stabilization of $D^{\mathbb{C}}_0-$spaces.

\begin{definition}
[The category $Sp_W^{D_0, U}$] A $D^{\mathbb{C}, W}_0-$spectrum $X$ is both a $D^{\mathbb{C}}_0-$spectrum and a $D^{\mathbb{C}}_0-$space in $D^{\mathbb{C}}_0T^W$. The category $Sp^{D_0, U}_{W}$ is the full
subcategory of $Sp^{D, U}$ consisting of $D^{\mathbb{C}, W}-$spectra.
\label{spd0w2}\end{definition}

\section{Global Model Structure}\label{modelgl}

In this section we will lift the model structures on orthogonal spectra in \cite{SS} to $Sp^{D_0}_W$.

The transfer theorem that we will use is Theorem 5.1, Chapter II, \cite{GPJJ}. We recall it below.

\begin{theorem}Let \begin{displaymath}
(F \dashv G) : \mathcal{C} \stackrel{\overset{G}{\leftarrow}}{\underset{F}{\to}}  \mathcal{D}
\end{displaymath}
be an adjoint pair and suppose $\mathcal{C}$ is a cofibrantly generated model category. Let $I$ and $J$ be chosen sets of generating
cofibrations and acyclic cofibrations, respectively.  Define a morphism $f: X\longrightarrow Y$ in $\mathcal{D}$ to be a weak equivalence or a fibration if $Gf$
is a weak equivalence or fibration in $\mathcal{C}$. Suppose further that

(1) the right adjoint $G: \mathcal{D}\longrightarrow \mathcal{C}$ commutes with sequential colimits; and

(2) every cofibration in $\mathcal{D}$ with the left lifting property with respect to all fibrations is a weak equivalence.

Then $\mathcal{D}$ becomes a cofibrantly generated model category. Furthermore the sets $\{Fi | i\in I\}$ and $\{Fj | j\in J\}$ generate the cofibrations and the acyclic cofribrations of $\mathcal{D}$ respectively.

\label{liftp}
\end{theorem}

We describe a strong level model structure on $Sp^{D_0}_W$.
\begin{definition}A map $f: X\longrightarrow Y$ in $Sp_W^{D_0}$ is
called a

-strong level equivalence if $Q(f)$ is a strong level equivalence in the category of orthogonal spectra, i.e. for each nonnegative  integer $n$ and each closed subgroup $H$ of $O(n)$,
$X(O(n), \mathbb{R}^n)^H\longrightarrow Y(O(n), \mathbb{R}^n)^H$ is a weak equivalence.

-strong level fibration if $Q(f)$ is a strong level fibration in the category of orthogonal spectra, i.e.  for each nonnegative integer $n$,
$X(O(n), \mathbb{R}^n)\longrightarrow Y(O(n), \mathbb{R}^n)$ is a $O(n)-$fibration.

-strong level cofibration if it is a morphism with the left lifting property with respect to all acyclic strong level fibrations.
 \label{strmodel}
\end{definition}

\begin{proposition}The strong level equivalences, strong level fibrations and the strong level cofibrations form a model structure, the strong level model structure
on $Sp^{D_0}_W$. It is cofibrantly generated. It is Quillen equivalent the strong level model structure on orthogonal spectra
in Proposition 4.3.7, \cite{SS}. \label{strlevel}\end{proposition}

\begin{proof}
Since $Q$ is the forgetful functor between diagram spectra, it commutes with colimits.

Since all the spaces are fibrant, all the objects in $Sp^{O}$ are fibrant in the strong level model structure. The terminal object in $Sp^{D_0}_W$ is the constant $D_0^W-$spectra.
$Y$ is a fibrant object in $Sp^{D_0}_W$ if and only if $Q(Y)$ is fibrant in $Sp^{O}$. The fibrant replacement functor in $Sp^{D_0}_W$ is the identity functor.
In addition, $Sp^{D_0}_W$ has functorial path objects. For each object $Y$, $Y^I$ is defined by \begin{equation}Y^I(G, V):=Y(G, V)^I. \end{equation}

Since the strong level model structure on orthogonal spectra is cofibrantly generated, by Theorem \ref{liftp}, the strong level equivalences, strong level fibrations and the strong level cofibrations form a model structure, which is cofibrantly generated.
 \end{proof}

 We also describe a global model structure on $Sp^{D_0}_W$.
\begin{definition}A map $f: X\longrightarrow Y$ in $Sp_W^{D_0}$ is
called a

-global equivalence if $Q(f)$ is a global equivalence in the category of orthogonal spectra, i.e. for each nonnegative  integer $n$ and each closed subgroup $H$ of $O(n)$,
$\pi^G_k(f): \pi^G_k(X)\longrightarrow \pi^G_k(Y)$ is a weak equivalence where $\pi^G_k(-)$ is the equivariant homotopy groups in Definition \ref{globalhtpygrp}.

-global fibration if $Q(f)$ is a global fibration in the category of orthogonal spectra, i.e. $f$ is strong level fibration and for each object $d=(O(V), V)$, $b=(O(W),
W)$ of $D$, any  closed subgroup $G$ of both $O(V)$ and $O(W)$,
the square \begin{equation}\xymatrix{X(d)^G\ar[d]^{f(d)^G} \ar[r]^>>>>>{\widetilde{\sigma}_{d, b}} &\mbox{map}^G(S^W, X(G, V\oplus W))\ar[d]^{\mbox{map}^G(S^W, f(G, V\oplus W))}\\
Y(d)^G\ar[r]_>>>>>{\widetilde{\sigma}_{d, b}} &\mbox{map}^G(S^W, Y(G, V\oplus W))} \label{glfib}\end{equation}
is a homotopy cartesian.

-global cofibration if it is a morphism with the left lifting property with respect to all acyclic global fibrations.
 \label{globalmodelsp}
\end{definition}

\begin{theorem}The global equivalences,  global fibrations and global cofibrations form a model structure, the strong level model structure
on $Sp^W_{D}$. It is cofibrantly generated. It is Quillen equivalent the global model structure on orthogonal spectra
in Theorem
4.3.18 in \cite{SS}. \label{globalmodelpf}\end{theorem}

\begin{proof}
Since $Q$ is the forgetful functor between diagram spectra, it commutes with colimits.

An object $Y$ in $Sp^{D_0}_W$ is fibrant if and only if $Q(Y)$ is fibrant in the global model structure of $Sp^{O}$. Let $R$ denote the fibrant replacement functor in $Sp^{O}$.
Then the functor \begin{equation}Y\mapsto P(R(QY)) \end{equation} is the fibrant replacement functor in $Sp^{D_0}_W$. Note that since $Q\circ P$ is the identity functor, $R(QY)$ is fibrant if and only if $P(R(QY))$ is fibrant.
In addition, $Sp^{D_0}_W$ has functorial path objects for fibrant objects.

Since the global model structure on orthogonal spectra is cofibrantly generated, by Theorem \ref{liftp}, the  global equivalences, global fibrations and global cofibrations form a model structure, which is cofibrantly generated.
 \end{proof}

\section{$\mathcal{F}-$global model structure}\label{dfspmodel}
Next we show our thoughts on the lift of $\mathcal{F}-$global model structure with $\mathcal{F}$ a global family of compact Lie groups, i.e. a non-empty class of compact Lie groups that is closed under isomorphisms, closed subgroups and quotient groups.

\begin{definition}Let $D^{\mathcal{F}}_0$ denote the full subcategory of $D_0$ with objects $(G, V)$ with $G\in\mathcal{F}$. \end{definition}
We use the symbol $D_0^{\mathcal{F}}T$ to denote the category of $D_0^{\mathcal{F}}-$spaces and $Sp^{D_0^{\mathcal{F}}}$ to denote the category of $D_0^{\mathcal{F}}-$spectra.

\begin{definition}[The category $D_0^{\mathcal{F}}T^W$]Define $D_0^{\mathcal{F}}T^W$ to be the full category of $D_0^{\mathcal{F}}T$ consisting of those objects $X: D_0^{\mathcal{F}}\longrightarrow T$ that maps each restriction map $(G, V)\longrightarrow (H, V)$ to an $H-$weak equivalence. We call the objects $D_0^{\mathcal{F}, W}-$spaces. \label{defd0wtf} \end{definition}

\begin{definition}
[The category $Sp_W^{D_0^{\mathcal{F}}}$] A $D_0^{\mathcal{F}, W}-$spectrum $X$ is both a $D_0^{\mathcal{F}}-$spectrum and a $D_0^{\mathcal{F}, W}-$space. The category $Sp_W^{D_0^{\mathcal{F}}}$ is the full
subcategory of $Sp^{D_0^{\mathcal{F}}}$ consisting of $D_0^{\mathcal{F}, W}-$spectra.
\label{spd0wf}\end{definition}

We have the question whether there is a model structure on $Sp_W^{D_0^{\mathcal{F}}}$ Quillen equivalent to that $\mathcal{F}-$global model structure
on orthogonal spectra
in Theorem 4.3.17 \cite{SS}. Below we show that when $\mathcal{F}$ consists of finite groups, we have a global model structure on $Sp_W^{D_0^{\mathcal{F}}}$.

We have a pair of adjoint functors between $Sp_W^{D_0^{\mathcal{F}}}$ and $Sp^O$. We have a functor $$l_F: L\longrightarrow
D_0^{\mathcal{F}}$$ sending $\mathbb{R}^n$ to $(\Sigma_n, \mathbb{R}^n)$. It gives a forgetful functor $U_F: Sp^{D_0^{\mathcal{F}}}\longrightarrow Sp^{O}$.
Let $$Q_F: Sp_W^{D_0^{\mathcal{F}}}\longrightarrow Sp^{D_0^{\mathcal{F}}}\longrightarrow Sp^{O}$$ denote the composition of the inclusion and the forgetful functor $U_F$.
And we can define a functor $P_F: Sp^{O}\longrightarrow Sp_W^{D_0^{\mathcal{F}}}$ by sending an orthogonal spectrum $Y$ to the $D_0^{\mathcal{F}, W}-$spectrum \begin{equation}
P_F(Y)(G, V):= Y(V); P_F(Y)(\phi_1, \phi_2):=Y(\phi_2) \label{leftPF}\end{equation} for any object $(G, V)$  and any morphism $(\phi_1, \phi_2)$ in $D_0^{\mathcal{F}}$.

\begin{lemma}
\begin{displaymath}
(P_F \dashv Q_F) : Sp^{O} \stackrel{\overset{Q_F}{\leftarrow}}{\underset{P_F}{\to}}  Sp_W^{D_0^{\mathcal{F}}}
\end{displaymath}
is a pair of adjoint functors. In addition, the unit $\eta^F: Id_X\longrightarrow Q_F\circ P_F$ is an isomorphism.
\label{adjpqf}
 \end{lemma}

Analogous to Proposition \ref{strlevel} and Theorem \ref{globalmodelpf}, we can construct the $\mathcal{F}-$level model structure
and the $\mathcal{F}-$global model structure on $Sp_W^{D_0^{\mathcal{F}}}$.
\begin{theorem} There is a model structure on $Sp_W^{D_0^{\mathcal{F}}}$ Quillen equivalent to the $\mathcal{F}-$level model structure on orthogonal spectra
in Proposition 4.3.5, \cite{SS}.

In addition, there is a model structure on $Sp_W^{D_0^{\mathcal{F}}}$ Quillen equivalent to the $\mathcal{F}-$global model structure on orthogonal spectra
in Theorem 4.3.17 \cite{SS}.
\end{theorem}

More generally, for any global family $\mathcal{F}$, we have the question whether there is a model structure on the category $Sp_W^{D_0^{\mathcal{F}}}$
Quillen equivalent to the $\mathcal{F}-$global model structure on orthogonal spectra.

\bigskip

Moreover, as shown in Section 6.4, \cite{Huanthesis}, there is a Reedy model structure on   $Sp_W^{D_0^{\mathcal{F}}}$
when $\mathcal{F}$ is the global family of finite groups.

We can define a degree
function on $D_0^{\mathcal{F}}$ by
$$deg(G, V)=|G|\dim V,$$ the order of the
group $G$ times the dimension of the vector space $V$. In proposition 6.16 \cite{Huanthesis}, we show that $D_0^{\mathcal{F}}$
is a dualizable generalized Reedy category in the sense of Definition 1.1 \cite{BM}, with the linear isometric embeddings raising the degree and
the restriction maps lowering the degree. Thus, there is a Reedy model structure on $Sp^{D_0^{\mathcal{F}}}$, moreover, on $Sp_W^{D_0^{\mathcal{F}}}$.

We describe roughly the Reedy model structure on  $Sp_W^{D_0^{\mathcal{F}}}$ below.

\begin{definition}
A map $f: X\longrightarrow Y$ in $Sp_W^{D_0^{\mathcal{F}}}$ is
 a

-Reedy cofibration if for each object $r$ in $D_0^{\mathcal{F}}$, the relative latching map\\
$X_r\coprod_{L_r(X)}L_r(Y)\longrightarrow Y_r$ is a
$Aut(r)-$cofibration.

\smallskip
-Reedy weak equivalence if for each $r$, $$f(r)^H:
X(r)^H\longrightarrow Y(r)^H$$ is a weak equivalence for each
 subgroup $H$ of $Aut(r)$.

\smallskip
-Reedy fibration if for each $r$, the relative matching map
$$X_r^H\longrightarrow M_r(X)^H\times_{M_r(Y)^H}Y_r^H$$ is a Serre
fibration for each subgroup $H$ of $Aut(r)$.
  \end{definition}
\begin{theorem}The Reedy cofibrations, Reedy weak equivalences and Reedy
fibrations form a model structure, the Reedy model structure, on
$Sp_W^{D_0^{\mathcal{F}}}$. \label{Reedy}
\end{theorem}

\section{Examples: Quasi-theories}\label{examples}

Quasi-elliptic cohomology is a variant of Tate K-theory. It is not an elliptic cohomology but does reflect the geometric features of the Tate curve. It has simpler form than
most elliptic cohomology theories and can be expressed explicitly by equivariant K-theories, which can be globalized. A thorough introduction of this theory can be found in
\cite{Huansurvey} and Chapter 2 \cite{Huanthesis}.

In \cite{HuanSpec} we construct a $\mathcal{I}_G-$FSP representing quasi-elliptic cohomology up to a weak equivalence, which, however, cannot arise from an orthogonal spectrum.
In \cite{Huanquasisp} we apply the idea to construct a $\mathcal{I}_G-$FSP weakly representing each quasi-theory $QE^*_{n, G}(-)$. In this section we show these $\mathcal{I}_G-$FSP together define a  $D^W_0-$spectrum.

\subsection{Quasi-theories}\label{quasith}
In this section we recall the quasi-theories. The main reference for that is \cite{Huanquasi}.

Let $G$ be a compact Lie group and $n$ denote a positive integer. Let $G^{tors}_{conj}$ denote a set of representatives of
$G-$conjugacy classes in the set $G^{tors}$ of torsion elements in $G$. Let $G^{n}_{z}$ denote set
$$\{\sigma=(\sigma_1, \sigma_2, \cdots \sigma_n )| \sigma_i\in G^{tors}_{conj}, [\sigma_i, \sigma_j]\mbox{   is  the   identity  element in  }G\}.$$

Let $\sigma=(\sigma_1, \sigma_2, \cdots \sigma_n)\in G^n_z$.  Define
\begin{align}C_G(\sigma)&:=\bigcap\limits_{i=1}^nC_G(\sigma_i); \label{Csigmadef}\\ \Lambda_G(\sigma)&:= C_G(\sigma)\times
\mathbb{R}^n/\langle (\sigma_1, -e_1), (\sigma_2, -e_2), \cdots (\sigma_n, -e_n)\rangle.\label{lambdadef}\end{align}
where $C_G(\sigma_i)$ is the centralizer of each $\sigma_i$ in $G$ and $\{e_1, e_2, \cdots e_n\}$ is a basis of $\mathbb{R}^n$.

Let $q:\mathbb{T}\longrightarrow U(1)$ denote the representation $t\mapsto e^{2\pi i t}$. Let $q_i=1\otimes\cdots\otimes q\otimes\cdots\otimes 1: \mathbb{T}^{n}\longrightarrow U(1)$ denote the tensor product with $q$ at the $i-$th position and trivial representations at other position. The representation ring $$R(\mathbb{T}^{n})\cong R(\mathbb{T})^{\otimes n}=\mathbb{Z}[q_1^{\pm}, \cdots q_n^{\pm}].$$

We have the exact sequence  \begin{equation}1\longrightarrow C_G(\sigma)\longrightarrow \Lambda_G(\sigma)\buildrel{\pi}\over\longrightarrow \mathbb{T}^{n}\longrightarrow 0 \end{equation}
where the first map is $g\mapsto [g, 0]$  and the second map is $\pi([g, t_1, \cdots t_n])=(e^{2\pi i t_1}, \cdots e^{2\pi it_n}).$
Then the map $\pi^*: R(\mathbb{T}^{n})\longrightarrow R\Lambda_G(\sigma)$ equips the representation ring $R\Lambda_G(\sigma)$ the structure as an
 $R(\mathbb{T}^{n})-$module.

This is Lemma 3.1 \cite{Huanquasi} presenting the relation between $RC_G(\sigma)$ and $R\Lambda_G(\sigma)$.

\begin{lemma} $\pi^*: R(\mathbb{T}^{n})\longrightarrow R\Lambda_G(\sigma)$ exhibits $R\Lambda_G(\sigma)$ as a free $R(\mathbb{T}^{n})-$module.

There is an $R(\mathbb{T}^{n})-$basis of $R\Lambda_G(\sigma)$
given by irreducible representations $\{V_{\lambda}\}$, such that
restriction $V_{\lambda}\mapsto V_{\lambda}|_{C_G(\sigma)}$ to $C_G(\sigma)$
defines a bijection between $\{V_{\lambda}\}$ and the set
$\{\lambda\}$ of irreducible representations of
$C_G(\sigma)$.\label{cl}\end{lemma}
\begin{definition}For equivariant cohomology theories $\{E_{H}^*\}_H$ and any $G-$space $X$, the corresponding quasi-theory $QE_{n, G}^*(X)$ is defined to be
$$\prod_{\sigma\in
G^{n}_{z}}E^*_{\Lambda_G(\sigma)}(X^{\sigma}).$$\label{qedef}\end{definition}

\begin{example}[Motivating example: Tate K-theory  and quasi-elliptic cohomology]
Tate $K-$theory is the generalized elliptic cohomology associated to the Tate curve. The elliptic cohomology theories form a sheaf of cohomology theories over the moduli stack of elliptic curves
$\mathcal{M}_{ell}$. Tate K-theory over Spec$\mathbb{Z}((q))$ is obtained when we restrict it to a punctured completed neighborhood of the cusp at $\infty$, i.e. the Tate
curve $Tate(q)$ over Spec$\mathbb{Z}((q))$  [Section 2.6,
\cite{AHS}].  The divisible group associated to Tate K-theory is $\mathbb{G}_m\oplus \mathbb{Q}/\mathbb{Z}$.

Other than the theory over Spec$\mathbb{Z}((q))$, we can define variants of Tate K-theory over Spec$\mathbb{Z}[q]$ and Spec$\mathbb{Z}[q^{\pm}]$ respectively. The theory over Spec$\mathbb{Z}[q^{\pm}]$
is of especial interest.
Inverting $q$ allows us to define a sufficiently non-naive equivariant cohomology theory and to interpret some constructions more easily in terms of extensions of groups over
the circle. The resulting cohomology theory is called quasi-elliptic cohomology \cite{Rez11}\cite{Huanthesis}\cite{Huansurvey}.
Its relation with Tate K-theory is \begin{equation}QEll^*_G(X)\otimes_{\mathbb{Z}[q^{\pm}]}\mathbb{Z}((q))=(K^*_{Tate})_G(X)
\label{tateqellequiv}\end{equation} which also reflects the geometric nature of the Tate curve.
The idea of quasi-elliptic cohomology is  motivated by Ganter's
construction of Tate K-theory  \cite{Dev96}.
It is not an elliptic cohomology  but a more robust and algebraically simpler treatment of
Tate K-theory. This new theory can be interpreted in a neat form
by equivariant K-theories. Some
formulations in it can be generalized to equivariant cohomology
theories other than Tate K-theory.

Quasi-elliptic cohomology $QEll^*_G(-)$ is exactly the quasi-theory $QK_{1, G}^*(-)$ in Definition \ref{qedef}.
\end{example}

\begin{example}[Generalized Tate K-theory and generalized quasi-elliptic cohomology]

In Section 2 \cite{Gan07} Ganter  gave an interpretation of $G-$equivariant Tate K-theory for finite groups $G$
by the loop space of a global quotient orbifold. Apply the loop construction $n$ times, we can get the $n-$th generalized Tate K-theory. The divisible group associated to it is
$\mathbb{G}_m\oplus (\mathbb{Q}/\mathbb{Z})^n$.

With quasi-theories, we can get a neat expression of it. Consider the quasi-theory
$$QK_{n, G}^*(X)=\prod_{\sigma\in
G^{n}_{z}}K^*_{\Lambda_G(\sigma)}(X^{\sigma}).$$
$QK_{n, G}^*(X)\otimes_{\mathbb{Z}[q^{\pm}]^{\otimes n}}\mathbb{Z}((q))^{\otimes n}$ is isomorphic to the $n-$th generalized Tate K-theory.

\label{generalizedquasi}
\end{example}

\subsection{Construction of the $D^{W}_0-$spectrum}\label{explicitconstr}
In this section we show the construction of a $D^{\mathbb{C}}-$spectrum representing the theory $QE^*_{n, G}(-)$. The details of the construction are in Section 5 \cite{Huanquasisp} with the main conclusion Theorem 5.18. Another reference is Chapter 4 \cite{Huanthesis}.

Let $G$ be a compact Lie group. We consider equivariant cohomology
theories $E$ that  have the same key features as equivariant
complex K-theories. More explicitly, \\ $\bullet$ The theories
$\{E_G^*\}_G$ have the change-of-group isomorphism, i.e. for any
closed subgroup $H$ of $G$ and $H-$space $X$, the change-of-group
map $\rho^G_H: E^*_G(G\times_HX)\longrightarrow E^*_H(X)$ defined
by $E^*_G(G\times_HX)\buildrel{\phi^*}\over\longrightarrow
E^*_H(G\times_H X)\buildrel{i^*}\over\longrightarrow E_H^*(X)$ is
an isomorphism where $\phi^*$ is the restriction map and $i:
X\longrightarrow
G\times_HX$ is the $H-$equivariant map defined by $i(x)=[e, x].$.\\
$\bullet$ There exists an orthogonal spectrum $E$ such that for
any compact Lie group $G$ and "large" real $G-$representation $V$
and a compact $G-$space $B$ we have a bijection
$E_G(B)\longrightarrow [B_+, E(V)]^G$. And $(E_G, \eta^{E}, \mu^{E})$ is the underlying orthogonal $G-$spectrum of $E$.\\
$\bullet$ Let $G$ be a compact Lie group and $V$ an orthogonal
$G-$representation. For every ample $G-$representation $W$, the
adjoint structure map $\widetilde{\sigma}^E_{V, W}:
E(V)\longrightarrow \mbox{Map}(S^W, E(V\oplus W))$ is a $G-$weak
equivalence.

In Theorem 5.18 \cite{Huanquasisp} we introduce the conclusion below.
\begin{theorem}
For each positive integer $n$ and each compact Lie group $G$, there is a well-defined functor $\mathcal{Q}_{G, n}$ from the category of orthogonal ring spectra to the category of
$\mathcal{I}_G-$FSP sending $E$  to
$(QE_n(G, -), \eta^{QE_n}, \mu^{QE_n})$ that weakly represents the
quasi-theory $QE_{n, G}^*$.
\label{igmain}
\end{theorem}

We sketch the construction of the $\mathcal{I}_G-$FSP
$(QE_n(G, -), \eta^{QE_n}, \mu^{QE_n})$ below. For more details please see \cite{Huanquasisp}.

\begin{itemize}

\item Let $V$ be a real $G-$representation.

\item Let $Sym(V):=\bigoplus\limits_{n\geq 0} Sym^n(V)$ denote the total
symmetric power. It inherits a $G-$action from that on $V$.

\item Let $S(G, V)_\sigma: = Sym(V)\setminus Sym(V)^\sigma=\bigcup\limits_{i=1}^n Sym(V)\setminus Sym(V)^{\sigma_i}$ for each $\sigma=(\sigma_1, \cdots \sigma_n)\in G^n_z$.

\item Let $E$ denote the unitary global spectrum representing the theory $E^*$.
\item $(V)^{\mathbb{R}}_\sigma$ is a specific real $\Lambda_G(\sigma)-$representation introduced in (A.2) in Appendix \cite{Huanquasisp}. We show the construction below.
\end{itemize}
 Let $\rho$ be a complex $G-$representation with
underlying space $V$. Let $i: C_G(\sigma)\hookrightarrow G$ denote
the inclusion. 
Let $\{\lambda\}$ denote all the irreducible complex
representations of $C_G(\sigma)$. We have
the decomposition of $C_G(\sigma)-$representation $i^*V$ into its isotypic components
$i^*V\cong\bigoplus\limits_{\lambda}V_{\lambda}$ where $V_{\lambda}$
denotes the sum of all subspaces of $V$ isomorphic to $\lambda$.
Each $V_{\lambda}=Hom_{C_G(g)}(\lambda,
V)\otimes_{\mathbb{C}}\lambda$ is unique as a subspace.

Each $V_{\lambda}$ can be equipped with a
$\Lambda_G(\sigma)-$action. Each $\lambda(\sigma_i)$ is of the form
$e^{\frac{2\pi i m_{\lambda i} }{l_i}}I$ with $l_i$ the order of $\sigma_i$, $0< m_{\lambda i}\leq l_i$
and $I$ the identity matrix. As shown in Lemma \ref{cl},
we have the well-defined complex
$\Lambda_G(\sigma)-$representations
\[(V_{\lambda})_{\sigma}:=V_{\lambda}\odot_{\mathbb{C}}
(q^{\frac{m_{\lambda 1}}{l_1}}\otimes\cdots\otimes q^{\frac{m_{\lambda n}}{l_n}})\] and
\begin{equation}(V)_{\sigma}:=\bigoplus_{\lambda}(V_{\lambda})_{\sigma}.\label{vgrepresentation}\end{equation}
In addition,
\begin{equation}(V)^{\mathbb{R}}_{\sigma} :=
(V\otimes_{\mathbb{R}}\mathbb{C})_{\sigma}\oplus
(V\otimes_{\mathbb{R}}\mathbb{C})^{*}_{\sigma}\label{realfaithv}\end{equation} is real
$\Lambda_G(\sigma)-$representation. It is faithful when $V$ is a faithful real $G-$representation.
\bigskip

From the ingredients above, we construct \begin{equation}F_{\sigma}(G, V):= \mbox{Map}_{\mathbb{R}}(S^{(V)^{\mathbb{R}}_\sigma}, E((V)^{\mathbb{R}}_\sigma\oplus
V^\sigma))\end{equation}
We have the unit
map $\eta_{\sigma}(G, V): S^{V^\sigma}\longrightarrow F_\sigma(G, V)$ and the
multiplication
$$\mu^F_{(\sigma, \tau)}((G, V), (H, W)): F_{\sigma}(G, V)\wedge F_{\tau}(H,
W)\longrightarrow F_{(\sigma, \tau)}(G\times H, V\oplus W)$$ cited in Proposition 5.2(i) \cite{Huanquasisp}.
\begin{equation}QE_{n, \sigma}(G, V):= \{t_1a+t_2b\in F_\sigma(G, V)\ast S(G, V)_\sigma|
\|b\|\leqslant t_2\}/\{t_1c_0+t_2 b\} \end{equation}where $\ast$ denotes the join.

\begin{equation}QE_n(G, V):=\prod\limits_{\sigma\in
G^{n}_{z}}\mbox{Map}_{C_G(\sigma)}(G, QE_{n, \sigma}(G, V))
\label{EGVnog}\end{equation}

 We
use $x_\sigma$ to denote the basepoint of $QE_{n, \sigma}(G, V)$.  For
each $v\in S^V$, there are $v_1\in S^{V^\sigma}$ and $v_2\in
S^{(V^\sigma)^{\perp}}$ such that $v=v_1\wedge v_2$. Let
$\eta^{QE_n}_{\sigma}(G, V): S^V\longrightarrow QE_{n, \sigma}(G, V)$ be the map
\begin{equation}\eta^{QE_n}_\sigma(G,
V)(v):=\begin{cases}(1-\|v_2\|)\eta_\sigma(G, V)(v_1)+ \|v_2\|v_2,
&\text{if $\|v_2\|\leqslant 1$;}\\ x_\sigma, &\text{if
$\|v_2\|\geqslant 1$.}
\end{cases}\label{finaleta}\end{equation} The unit map is defined by \begin{equation}\eta^{QE_n}(G, V): S^V\longrightarrow
\prod\limits_{\sigma\in G^{n}_{z}}\mbox{Map}_{C_G(\sigma)}(G, QE_{n, \sigma}(G,
V))\mbox{,   } v \mapsto \prod\limits_{\sigma\in
G^{n}_{z}}(\alpha\mapsto \eta^{QE_n}_{\sigma}(G, V)(\alpha\cdot
v)),\label{GetaQEll}
\end{equation}

Next, we construct the multiplication
map $\mu^{QE_n}$. First we define a map $$\mu^{QE_n}_{(\sigma, \tau)}((G,V),
(H, W)): QE_{n, \sigma}(G, V)\wedge QE_{n, \tau}(H, W)\longrightarrow QE_{n, (\sigma,
\tau)}(G\times H, V\oplus W)$$ by sending a point $[t_1a_1+
t_2b_1]\wedge [u_1a_2 + u_2b_2]$ to
\begin{equation}
\begin{cases}[(1-\sqrt{t^2_2+u^2_2})\mu^F_{(\sigma, \tau)}((G,V)
, (H, W))(a_1\wedge a_2) &\text{if $t^2_2+u^2_2\leq 1$ and
$t_2u_2\neq 0$;}\\ + \sqrt{t^2_2+u^2_2}(b_1+b_2)],
& \\
[(1-t_2)\mu^F_{(\sigma, \tau)}((G,V), (H, W))(a_1\wedge a_2)+ t_2b_1],
&\text{if $u_2=0$ and $0<t_2<1$}; \\ [(1-u_2)\mu^F_{(\sigma, \tau)}((G,V), (H, W))(a_1\wedge a_2) + u_2b_2],
&\text{if $t_2=0$ and $0<u_2<1$; } \\ [1\mu^F_{(\sigma, \tau)}((G,V), (H, W))(a_1\wedge a_2) + 0],
&\text{if $u_2=0$ and $t_2=0$; }\\ x_{\sigma, \tau},
&\text{Otherwise.}\end{cases}\label{muEg}
\end{equation} where $x_{\sigma, \tau}$ is the basepoint of
$QE_{n, (\sigma, \tau)}(G\times H, V\oplus W)$.

The basepoint of $QE_n(G, V)$ is the product of the basepoint of
each factor $\mbox{Map}_{C_G(\sigma)}(G, QE_{n, \sigma}(G, V))$, i.e. the product
of the constant map to the basepoint of  each $QE_{n, \sigma}(G, V)$.

We can define the multiplication $\mu^{QE_n}((G, V), (H, W)): QE_n(G,
V)\wedge QE_n(H, W)\longrightarrow QE_n(G\times H, V\oplus W)$ by
$$\big(\!\!\prod\limits_{\sigma\in G^{n}_{z}}
\!\!\alpha_\sigma\big)\wedge\big(\!\!\prod\limits_{ \tau\in
H^{n}_{z}}\!\!\beta_{\tau}\big)\mapsto\prod\limits_{\substack{\sigma\in
G^{n}_{z}\\ \tau\in H^{n}_{z}}}\!\!\bigg(\!\!(\sigma',
\tau')\mapsto \mu^{QE_n}_{(\sigma, \tau)}((G, V), (H,
W))\big(\alpha_\sigma(\sigma')\wedge\beta_\tau(\tau')\big)\!\!\bigg).
$$

However, the equivariant orthogonal spectra cannot arise from a global spectrum in the sense of Remark 4.1.2, \cite{SS}.

\subsection{$QE_n$ is a global spectrum in $Sp^W_{D_0}$}\label{QEnsp}

Via a completely parallel proof to that of Theorem 7.1.5, \cite{Huanthesis}, we have the conclusion below.
\begin{theorem}The $\mathcal{I}_G-$FSP $QE_n:\mathcal{I}_G\longrightarrow G\mathcal{T}$ together with the
unit map $\eta^{QE_n}$ defined in (\ref{GetaQEll}) and  the
multiplication $\mu^{QE_n}((G, -), (G, -))$
 that weakly represents the quasi-theory $QE^*_{n, G}(-)$ with $G$ varying all the finite groups can fit together to give
a $D^W_0-$spectrum. \end{theorem}
First we show the unstable conclusion.
\begin{theorem}$QE_n$ is a $D_0^W-$space. \label{qellrd0}\end{theorem}
\begin{proof}
Let $(G, V)$ and $(H, W)$ be two objects and $\phi=(\phi_1, \phi_2): (G, V)\longrightarrow (H, W)$  be a
morphism in $D_0$. Let $H'$ denote $H\cap
O(\phi_2(V))$. Recall both $\phi_1$ and $\phi_2$ are
injective. Note that for $\tau=(\tau_1, \cdots \tau_n)\in H^n_z$, $\phi_1(\tau)=(\phi_1(\tau_1), \cdots \phi_1(\tau_n))$ can be in $G^n_{z}$. $\phi_2$ gives the
linear isometric embedding
\begin{equation}\phi_{2  \ast}:         (V)^{\mathbb{R}}_{\phi_1(\tau)}\oplus V^{\phi_1(\tau)}\longrightarrow
(W)^{\mathbb{R}}_\tau\oplus W^\tau  \mbox{,          } (v_1, v_2) \mapsto
(\phi_2(v_1), \phi_2(v_2)).\end{equation} It is equivariant in the
sense that
$$\phi_{2  \ast}([\phi_1(a), t]\cdot x)=[a, t]\cdot\phi_{2
\ast}(x)$$   for  any    $x\in (V)_{\phi_1(\tau)}^{\mathbb{R}}\oplus V^{\phi_1(\tau)}$, $[a,
t]\in\Lambda_{H'}(\tau)$,  and  $[\phi_1(a),
t]\in\Lambda_G(\phi_1(\tau)).$

Let $$\beta: S^{(V)^{\mathbb{R}}_{\phi_1(\tau)}}\longrightarrow E((V)_{\phi_1(\tau)}^{\mathbb{R}}\oplus
V^{\phi_1(\tau)})$$ be an element in $F_{\phi_1(\tau)}(G, V)$. We can
define $$F(\phi)_{\tau}: \mbox{Map}_{\mathbb{R}}(S^{(V)^{\mathbb{R}}_{\phi_1(\tau)}}, E((V)^{\mathbb{R}}_{\phi_1(\tau)}\oplus
V^{\phi_1(\tau)}))\longrightarrow \mbox{Map}_{\mathbb{R}}(S^{(W)^{\mathbb{R}}_\tau }, E
((W)_\tau^{\mathbb{R}}\oplus W^\tau))$$ with each $F(\phi)_{\tau}(\beta)$ defined by the composition:

$S^{(W)^{\mathbb{R}}_\tau} = S^{(W)^{\mathbb{R}}_\tau-(V)^{\mathbb{R}}_{\phi_1(\tau)}}\wedge S^{(V)^{\mathbb{R}}_{\phi_1(\tau)}}
\buildrel{Id\wedge \beta}\over\longrightarrow
S^{(W)_\tau^{\mathbb{R}}-(V)^{\mathbb{R}}_{\phi_1(\tau)}}\wedge E((V)_{\phi_1(\tau)}^{\mathbb{R}}\oplus
V^{\phi_1(\tau)})
\longrightarrow E((W)_{\tau}^{\mathbb{R}}\oplus V^{\phi_1(\tau)})
\buildrel{E(Id\oplus\phi_2)}\over\longrightarrow
E((W)_\tau^{\mathbb{R}}\oplus W^\tau)$

where the third map is the structure map of $E$ and
$Id\oplus\phi_2$ is the evident linear isometric embedding
$$(W)_{\tau}^{\mathbb{R}}\oplus V^{\phi_1(\tau)}
\longrightarrow (W)^{\mathbb{R}}_\tau\oplus W^\tau.$$ It is $\mathbb{R}-$linear.

It's straightforward to check for morphisms $$\begin{CD}(G, V) @>
{\phi=(\phi_1, \phi_2)}>> (H, W) @> {\psi=(\psi_1, \psi_2)}>> (K,
U)\end{CD}$$ So we have \begin{equation}F(\psi)_{\delta}\circ
F(\phi)_{\psi_1(\delta)}= F(\psi\circ\phi)_{\delta}\mbox{  for any  }\delta\in K\cap O(\phi_2\circ\psi_2(V)).\end{equation} and if $\phi$ is the identity map,
$F(\phi)_{\tau}$ is identity for each $\tau$.

$\phi_2$  gives  embeddings
$$Sym(V)\longrightarrow Sym(W)\mbox{        and         } Sym(V)^{\phi_1(\tau)}\longrightarrow
Sym(W)^\tau,$$ So we have well-defined
\begin{equation}S(\phi)_{\tau}: S(G, V)_{\phi_1(\tau)}\longrightarrow S(H, W)_{\tau}\mbox{,               }\end{equation} which is equivariant in the sense:
$$S(\phi)_{\tau}(\phi_1(\tau')\cdot y)=\tau'\cdot S(\phi)_{\tau}(y)\mbox{,     for  any    } \tau'\in C_{H'}(\tau).$$

Then we have the join $$F(\phi)_{\tau}\ast S(\phi)_{\tau}:
F_{\phi_1(\tau)}(G, V)\ast S(G, V)_{\phi_1(\tau)}\longrightarrow
F_\tau(H, W)\ast S(H, W)_{\tau}.$$

Let $$\phi_{\star}: QE_{n, \phi_1(\tau)}(G, V)\longrightarrow QE_{n, \tau}(H,
W)$$ denote the quotient of the restriction of $F(\phi)_{\tau}\ast S(\phi)_{\tau}$.

In addition, for each $\tau\in H^{n}_z$,
 each $f\in \mbox{Map}_{C_G(\phi_1(\tau))} (G, QE_{n, \phi_1(\tau)}(G, V))$, we can define
$$\widetilde{\phi_{\star}}(f):
C_{H}(\tau)\times_{C_{H'}(\tau)} H'\longrightarrow
\mbox{Map}_{C_H(\tau)} ( G, QE_{n, \tau}(H, W))$$ by
\begin{equation}\widetilde{\phi_{\star}}(f)([\alpha, \tau'])=\alpha\cdot (\phi_{\star}\circ f\circ\phi_1)(\tau')\end{equation} for each
$[\alpha, \tau']\in C_{H}(\tau)\times_{C_{H'}(\tau)} H'$.

Then we define $$QE_n(\phi)_\tau: \mbox{Map}_{C_G(\phi_1(\tau))} ( G,
QE_{n, \phi_1(\tau)}(G, V))\longrightarrow \mbox{Map}_{C_H(\tau)} ( H,
QE_{n, \tau}(H, W))$$ by
\begin{equation}QE_n(\phi)_\tau(f)(g):=\begin{cases}\widetilde{\phi_{\star}}(f)(g),&\text{if $g\in C_{H}(\tau)\times_{C_{H'}(\tau)}
H'$;}\\
c_0,&\text{otherwise.}\end{cases}\end{equation} where $c_0$
denotes the basepoint of $QE_{n, \tau}(H, W)$.

Finally, we define
\begin{equation}QE_n(\phi):=\prod\limits_{\tau\in H^{n}_z}QE_n(\phi)_\tau: QE_n(G, V)\longrightarrow QE_n(H, W).\end{equation} For
those $\sigma\in G^{n}_z$ not conjugate to any element in $H^n_z$, the factor $\mbox{Map}_{C_G(\sigma)} ( G,
QE_{n, \sigma}(G, V))$ has no effect on the image.

If $\phi$ is the identity map, $QE_n(\phi)$ is the identity map. If
$\phi: (G, V)\longrightarrow (H, W)$ and $\psi: (H,
W)\longrightarrow (K, U)$ are two morphisms in $D$,
we have
$$QE_n(\psi\circ\phi)=QE_n(\psi)\circ QE_n(\phi). $$ 
So
$QE_n$ defines a $D^W_0-$space.

\end{proof}

\begin{theorem}$QE_n$ is a $D_0-$FSP weakly representing $QE_n^*$. Especially, $QE_n$ is a global ring spectrum in $Sp_W^{D_0}$.\label{ERFSPoverS}\end{theorem}
The proof is analogous to the proof of Theorem 7.2.3 \cite{Huanthesis}. It is straightforward and left to the readers.

\begin{remark}Given a unitary global spectrum representing $E$, we can also construct a unitary $D_0^W-$spectrum representing $QE^*_n$ in the same way described in this section. The complex $\Lambda_G(\sigma)-$representations needed in the construction is constructed in Section A.1 \cite{Huanquasisp}.\end{remark}

\begin{remark}In \cite{SSST} Stapleton discussed the generalized Morava E-theories \cite{SSST} $$E_n^*(L^h(X/\!\!/G))$$ where $L$ is the $p-$adic loop functor $$L(X/\!\!/G):=Hom_{top.gpd}(\ast/\!\!/\mathbb{Z}_p, X/\!\!/G)\cong (\coprod_{\alpha\in Hom(\mathbb{Z}_p, G)} X^{im \alpha})/\!\!/ G,$$ and $h$ is a positive integer. In the case when $\{E^*_G(-)\}_G$ are equivariant cohomology theories satisfying the conditions at the beginning of
 Section \ref{explicitconstr}, we can construct a $\mathcal{I}_G-$FSP weakly representing each $E^*(L^h(X/\!\!/G))$. The construction is analogous to that of $QE_h(G, -)$ and contains many details and symbols. Theorem \ref{igmain} guarantees that  these $\mathcal{I}_G-$FSP together give a $D_0^W-$spectrum representing $E^*(L^h(-))$, which means it can be globalized in the almost global homotopy theory.

 If $E^0(B\mathbb{Z}/p)$ is the ring of functions on the $p-$torsion of the formal group $\mathbb{G}_E$ of $E^*(-)$, then the $p-$divisible group associated to $E^*(L^h(-))$ is
 $$\mathbb{G}_E\oplus (\mathbb{Q}_p/\mathbb{Z}_p)^h.$$
\label{gmoravasp}
\end{remark}

\begin{remark}In this paper we show that from one global ring spectrum $E$, we can construct infinitely many almost global spectrum $\{QE_n\}_{n\in \mathbb{N}}$. So there is a large family of examples for global homotopy theory. Moreover, in light of Remark \ref{gmoravasp}, I have the question  whether the conjecture is true that the globalness of a cohomology theory is completely determined the formal component of its divisible group and when the $\acute{e}$tale component of it varies the globalness does not change.\label{etaless}\end{remark}




\end{document}